\documentclass[12pt,reqno,a4paper, english]{amsart}

\usepackage[top=3.5cm,bottom=3cm,left=3cm,right=3cm]{geometry}
\usepackage{changepage}
\usepackage{placeins}
\usepackage{amsmath, amsthm, amsfonts, amssymb,enumerate}
\usepackage{bm} 
\usepackage{amsthm} 
\usepackage{mathrsfs} 
\usepackage{mathtools}
\usepackage{slashed}
\usepackage{dsfont} 
\usepackage{upgreek}
\usepackage{nccmath}
\usepackage{stmaryrd}

\usepackage[dvipsnames]{xcolor} 
\usepackage{tikz}
\usepackage{microtype} 
\usepackage{comment} 
\usepackage[titletoc,title]{appendix}
\usepackage[font=small]{caption} 
\usepackage{systeme} 
\usepackage{enumitem} 
\usepackage{stmaryrd}

\usepackage{pifont} 

\usepackage{etoolbox}

\usepackage{pgfplots}

\usepackage{hyperref}
\hypersetup{colorlinks={true},linkcolor={awesome},citecolor={greenBleu},urlcolor={britishracinggreen}}
\definecolor{awesome}{rgb}{1, 0.1, 0}
\definecolor{amber}{rgb}{1.0, 0.75, 0.0}
\definecolor{greenBleu}{rgb}{0.0, 0.6, 0.5}
\definecolor{britishracinggreen}{rgb}{0.0, 0.26, 0.15}
\definecolor{caputmortuum}{rgb}{0.7, 0.55, 0.5}

\makeatletter
\newcommand\footnoteref[1]{\protected@xdef\@thefnmark{\ref{#1}}\@footnotemark}
\makeatother



\theoremstyle{plain}
\newtheorem{theorem}{Theorem}[section]
\newtheorem*{theorem*}{Theorem}
\newtheorem{prop}[theorem]{Proposition}
\newtheorem*{prop*}{Proposition}

\newtheorem{corollary}[theorem]{Corollary}
\newtheorem*{corollary*}{Corollary}
\newtheorem{lemma}[theorem]{Lemma}
\newtheorem*{lemma*}{Lemma}

\theoremstyle{definition}

\theoremstyle{remark}
\newtheorem{Rq}{Remark}[section]

\numberwithin{equation}{section}

\DeclareMathOperator{\eee}{e}

\DeclareMathOperator{\Id}{Id}

\newcommand{\be}{\begin{equation}}
\newcommand{\ee}{\end{equation}}
\newcommand{\bes}{\begin{equation*}}
\newcommand{\ees}{\end{equation*}}

\newcommand{\R}{\mathbb{R}} 
\newcommand{\C}{\mathbb{C}} 
\newcommand{\N}{\mathbb{N}} 
 
\newcommand{\T}{\mathbb{T}}

\newcommand{\Nzero}{\mathbb{N}_{\geq0}}

\newcommand{\qlq}{\forall}
\newcommand{\e}{\exists}
\newcommand{\bk}{\backslash}
\newcommand{\p}{\partial}

\newcommand{\eps}{\varepsilon}

\newcommand{\del}{\delta}

\newcommand{\Ltwo}{L^2_+(\T)}
\newcommand{\LtwoR}{L^2_+(\R)}

\newcommand{\lracc}[1]{\left\{#1\right\}} 

\newcommand{\ueps}{u^\eps}

\newcommand{\zd}{ZD_{\pm}[u_0]}
\newcommand{\zdp}{ZD_{+}[u_0]}
\newcommand{\zdm}{ZD_{-}[u_0]}

\newcommand{\vertiii}[1]{{\vert\kern-0.25ex\vert\kern-0.25ex\vert #1 
    \vert\kern-0.25ex\vert\kern-0.25ex\vert}}

\newcommand{\va}[1]{\lvert#1\rvert}

\newcommand{\ps}[2]{\left\langle #1\,|\, #2  \right\rangle}

\newcommand{\fr}{\widehat}

\let\oldtocsection=\tocsection

\let\oldtocsubsection=\tocsubsection

\let\oldtocsubsubsection=\tocsubsubsection

\renewcommand{\tocsection}[2]{\hspace{0em}\oldtocsection{#1}{#2}}
\renewcommand{\tocsubsection}[2]{\hspace{2em}\oldtocsubsection{#1}{#2}}
\renewcommand{\tocsubsubsection}[2]{\hspace{4.7em}\oldtocsubsubsection{#1}{#2}}

\pgfplotsset{compat=1.18}

\begin{document}


\title[Zero dispersion limit of (CM)]{Zero dispersion limit of the Calogero--Moser derivative NLS equation}
\author{Rana Badreddine}
\address{Universit\'e Paris-Saclay, Laboratoire de math\'ematiques d'Orsay, UMR 8628 du CNRS, B\^atiment 307, 91405 Orsay Cedex, France}
	\email{\href{mailto: rana.badreddine@universite-paris-saclay.fr}{rana.badreddine@universite-paris-saclay.fr} }

	\date{\today}
\maketitle

\begin{abstract}
    We study the zero--dispersion limit of the Calogero--Moser derivative NLS equation
    $$
          i\p_tu+\p_x^2 u \,\pm\,2D\Pi(|u|^2)u=0 \,, \qquad x\in\R\,,
    $$
    starting from an initial data $u_0\in L^2_+(\R)\cap L^\infty (\R)\,,$ where $D=-i\p_x\,,$ and $\Pi$ is the Szeg\H{o} projector defined as  $\fr{\Pi u}(\xi)=\mathds{1}_{[0,+\infty)}(\xi)\fr{u}(\xi)\,.$
    We characterize the zero--dispersion limit solution by an explicit formula. Moreover, we identify it, in terms of the branches of the multivalued solution of the inviscid Burgers--Hopf equation. Finally, we infer that it satisfies a maximum principle.
\end{abstract}

\tableofcontents

\section{Introduction}
We consider a nonlocal nonlinear Schrödinger equation called \textit{the Calogero--Moser derivative nonlinear Schrödinger equation}
\be\label{CM}\tag{CM}
    i\p_tu+\p_x^2 u \,\pm\,2D\Pi(|u|^2)u=0 \,, \qquad x\in\R\,,
\ee
where $D=-i\p_x\,,$ and $\Pi\equiv \Pi_+$ is the Szeg\H{o} projector
\[
\Pi u (x)=\int_\R \frac{u(y)}{y-x}\frac{\mathrm{d}y}{2\pi i}\,,
\]
which is an orthogonal projector from $L^2(\R)$ into the Hardy space 
\begin{align*}
    L^2_+(\R):=&\,\{u \in L^2(\R)\, ,\, \operatorname{supp} \fr{u}\subseteq [0,+\infty[\,\}
    \\
    \cong&\,\{u\in \operatorname{Hol}(\C_+)\,,\, \sup_{y>0}\int_{\R}\va{u(x+iy)}^2 dx<+\infty\}\,,
\end{align*}
with $\C_+:=\lracc{z\in\C\,;\,\operatorname{Im}(z)>0}$\,.
Typically, in Fourier transform, $\Pi$ can be read as
\be\label{Szego proj}
    \fr{\Pi u}(\xi)=\mathds{1}_{[0,+\infty)}(\xi)\,\fr{u}(\xi)\,.
\ee
This equation comes in two versions: one with the ``$+$" sign in front of the nonlinearity, referring to the focusing equation, and the other with the ``$-$" sign, referring to the defocusing equation. Through this paper, the $\pm$ and $\mp$ symbols will be interchanged based on the following rule: the upper sign will correspond to the focusing case and the lower sign to the defocusing case.
\vskip0.25cm
It is known since the work of \cite{GL22} that the focusing \eqref{CM}--equation is globally well--posedness in $H^k_+(\R):=H^k(\R)\cap L^2_+(\R)\,,$  $k\in\N_{\geq 1}$ for small initial data ($\|u_0\|_{L^2}<\sqrt{2\pi}$)\,. 
This was achieved by establishing a uniform $H^k$--bound of the solution $u(t)$ over time. The same line of arguments enable the global well--posedness of the defocusing equation in $H^k_+(\R)\,,$ $k\in\N_{\geq 1}$
for any initial data $u_0\,.$
Subsequently, the case on the torus ($x\in\T$) has been investigated by the author under the name of \textit{Calogero–Sutherland
DNLS equation} \cite{Ba23a, Ba23b}\,, where  the GWP has been established in $H^s_+(\T)$\,, $s\geq 0$ in the focusing and defocusing cases, with small initial data in the focusing case.
Later, \cite{KLV23} extended the GWP results on $\R\,,$ from the high regularity spaces \cite{GL22} up to the scaling--critical space $L^2_+(\R)\,.$ 
More recently, \cite{HK24} established the blow--up of the $H^s$--norm's solution, in a time $T \in(0,+\infty]$ for initial data $u_0\in H^\infty_+(\R)$ satisfying $\|u_0\|_{L^2}^2=2\pi+\eps\,,$ for any $\eps>0\,.$
\vskip0.25cm

From a physical standpoint, the scenarios described by the Calogero--Moser DNLS equation share notable similarities with the \textbf{Benjamin--Ono equation}. In both cases, they characterize weakly nonlinear dispersive internal waves located at the interface between two fluid layers of different densities, with the lower layer having infinite depth \cite{BLS08,Sa19,Pe95}. In the context of the Benjamin--Ono equation, the solution delineates the progression of these internal waves. On the other hand, concerning the \eqref{CM}--equation, it illustrates a model for the envelope of approximately monochromatic waves within the aforementioned settings.

\vskip0.25cm

Recently, G\'erard \cite{Ge23} studied the Benjamin--Ono equation with small dispersion $\eps>0\,,$ described as
\be\label{Bo-eps}\tag{BO-eps}
\begin{cases}
    \p_t u^\eps +\p_x((u^\eps)^2)=\eps \va{D}\p_x u^\eps
    \\
     u^\eps|_{t=0}=u_0
\end{cases}\,,
\ee
where he established that the weak limit in $L^2$ of $u^\eps\,,$ as $\eps$ approaches $0\,,$ is characterized in terms of the branches of the multivalued solution of the \textbf{inviscid Burgers--Hopf equation}. Novelty, this characterization holds for any $u_0\in L^2(\R)\,,$ with $u_0$ is a $C^1$ function tending to $0$ at infinity as well as its first derivative \cite{Ge23}\,. Observe that one can discern the emergence of the Burgers equation by formally taking the limit in \eqref{Bo-eps}  as $\eps\to 0$\,.
The act of neglecting the dispersion component in the equation is commonly acknowledged in the literature as the ``zero-dispersion limit" or ``semiclassical limit". In the following, we will use the terminology of `zero-dispersion limit"\,. Additionally,  we will refer to the weak $L^2$--limit of $u^\eps$ when $\eps\to0$\,, as ``the weak zero--dispersion limit solution''.  It's important to note that the selection of initial data, represented by $u_0$, remains independent of $\eps$\,.

\vskip0.25cm

In this paper, we propose to investigate the zero--dispersion limit problem for the Calogero--Moser DNLS equation.
Thus, we consider the rescaled version of \eqref{CM} with small dispersion  $\eps>0\,,$
\be\label{CM-eps-intro}\tag{CM-eps}
\begin{cases}
    i\p_t\ueps\,+\eps\,\p_x^2 \ueps \,\pm\,2D\Pi(|\ueps|^2)\ueps=0
    \\
     u^\eps|_{t=0}=u_0
\end{cases}\,.
\ee
The aim is to write the weak limit in $L^2$ of the solution $\ueps$ of \eqref{CM-eps-intro}\,, as $\eps\to 0$\,,
in terms of the branches of the multivalued solution of the Burgers equation. However, here, it is less evident compared to the Benjamin--Ono case, why the Burgers equation emerges in this context. 
For this purpose, observe that when formally taking $\eps\to 0\,,$ the \eqref{CM-eps-intro} becomes
\begin{equation}
        \label{CM-zero}\tag{CM--zero}
            i\p_tu\,\pm\,2D\Pi(|u|^2)u=0\,.
\end{equation}
Consequently,  if $u$ solves the previous equation, then $\boldsymbol{v}=\va{u}^2$ solves the Burgers equation
\be\label{Burgers equation}
        \p_t \boldsymbol{v}=\pm 2\boldsymbol{v}\, \p_x \boldsymbol{v}
\ee       
as
    \begin{align*}
        \p_t \boldsymbol{v}
        =&2\operatorname{Re}(\p_tu \bar{u})
        \\
        =&\pm4\operatorname{Re}(\p_x\Pi(\va{u}^2)\va{u}^2)\notag
        \\
        =&\pm2\,\big(\p_x\Pi(\va{u}^2)+\overline{\p_x\Pi(\va{u}^2)} \big)\va{u}^2\notag
        \\
        =&\pm\p_x \va{u}^4\ =\,\pm \p_x \boldsymbol{v}^2\,\notag
        \\
        =&\pm 2\boldsymbol{v}\, \p_x \boldsymbol{v}\notag\,.
\end{align*}


\vskip0.25cm


But before proceeding, it is essential to prove the existence of a weak zero dispersion limit for \eqref{CM}\,. The upcoming theorem seeks to establish this existence and even to characterize this $L^2$--weak limit explicitly as an element of the Hardy space.
The notation $\zdp$ represents the weak zero-dispersion limit solution in the focusing case for \eqref{CM}, and $\zdm$ corresponds to the one in the defocusing case.

\begin{theorem}\label{weak limit theorem zero dispersion}
    Given an initial data $u_0\in\LtwoR\cap L^\infty(\R)$  (with $\|u_0\|_{L^2}< \sqrt{2\pi}$~\footnote{The constant $\sqrt{2\pi}$ is to ensure the GWP of \eqref{CM} in the focusing case.} in the focusing case), the weak (in $L^2$--space) zero-dispersion limit solution $\zd$ of \eqref{CM} exists, and is characterized via the following explicit formula
    \be\label{explicit formula R zero dispersion-intro}
        \zd(t,z)=\Big(\operatorname{Id}\,\mp\,2tT_{u_0}T_{\bar{u}_0}(X^*-z)^{-1}\Big)^{-1}u_0(z)\,, \qquad t\in\R \,,\, z\in\C_+\,,
    \ee
    where the operators $T_{v}$ and  $X^*$  are defined respectively at \eqref{Tv} and \eqref{X*}\,.
    In addition, we have
    \[
    \|\zd(t)\|_{L^2}\leq \|u_0\|_{L^2}\,.
    \]
    Furthermore,
    if $u_0^n\to u_0$ strongly in $L^2$ as $n\to \infty\,,$ with $\sup_n\|u_0^n\|_{L^\infty}<+\infty$\,,  then for all $T>0$,
    \[
        \sup_{t\in[-T,T]}\va{ZD_{\pm}[u_0^n](t)- \zd(t)}\underset{n\to \infty}{\xrightharpoonup{\hspace*{0.5cm}}} 0 \ \text{in } L^2(\R)\,.
    \]
\end{theorem}
\vskip0.25cm
Usually, when considering the scenario of zero-dispersion limit, the emergence of shocks can be observed. These shocks manifest as we begin to neglect dispersive effects, allowing the nonlinear term to dominate. 
With the existence of the weak zero-dispersion limit established in the previous theorem, our objective in the following theorem is to highlight these shocks, by 
addressing the connection between this zero-dispersion limit solution of \eqref{CM} and the branches of the multivalued solution of the inviscid Burgers equation, which is known for its tendency to exhibit shock formations.

\begin{theorem}\label{Theorem solution mutivalué- intro}
    Let $u_0\in L^2_+(\R)$ (with $\|u_0\|_{L^2}< \sqrt{2\pi}$\, in the focusing case), such that  $u_0$ is a $C^1$ function tending to $0$ at infinity, with a bounded derivative in $L^\infty(\R)\,$.\footnote{~Note that any function in $H^{s}_+(\R):=H^{s}(\R)\cap \LtwoR\,,$ $s>\frac32$ satisfies these conditions.} Then, for every time $t\in\R\,,$ 
    and for almost every $x\in\R\,,$ the algebraic equation
    \be\label{sol of the branches}
        y\mp2t\va{u_0(y)}^2=x
    \ee
    has an odd number of simple real solutions
    $
        y_0:=y_0(t,x)\,<\, \ldots\,<\, y_{2\ell}:=y_{2\ell}(t,x)\,,
    $
    and the zero--dispersion limit of \eqref{CM} is given 
    by
    \be\label{zd[u0]-multivalue solution-intro}
        \zd (t,x)=\mathrm{e}^{i \varphi(t,x)}\left(\mp i\, \frac{\va{t}}{t}\right)^\ell \prod_{k=0}^{2\ell}\va{u_0(y_k)}^{(-1)^k} ,
\ee
where 
\[
    \varphi(t,x)=\arg(u_0(x))+\frac{1}{2\pi}\int_0^{+\infty}
    \frac{1}{s}\log\left(\frac{\ s\mp 2t \va{u_0(x+s)}^2}{-s\mp 2t \va{u_0(x-s)}^2}\frac{\prod_{k=0}^{2\ell}(x-s-y_k)}{\prod_{k=0}^{2\ell}(x+s-y_k)}\right)\mathrm{d}s\,.
\]
\end{theorem}

\begin{Rq}
    When $\ell>0$\,, then any solution $y_k:=y_k(t,x)$\,, $k\in\lracc{0\,,\ldots\,, 2\ell}$ satisfying the algebraic equation~\eqref{sol of the branches}, represents a branch of the multivalued solution of the Burgers equation \eqref{Burgers equation} at a time $t$ beyond the shock time, and at a position $x\,.$ 
\end{Rq}
\begin{Rq}
    Observe that if we start from a rational initial data $u_0\,,$ then in view of identity \eqref{zd[u0]-multivalue solution-intro}\,, the weak zero--dispersion limit is also a rational function.  This outcome does not seem to be evident by solely examining the identity \eqref{explicit formula R zero dispersion-intro} obtained in Theorem~\ref{weak limit theorem zero dispersion}\,.
\end{Rq}

\begin{Rq}
By taking the modulus of \eqref{zd[u0]-multivalue solution-intro}\,, we deduce
\[
    \log\va{\zd(t,x)}=\sum_{k=0}^{2\ell} (-1)^k \log\va{u_0(y_k)}\,.
\]
This result should be compared to the one obtained by  \cite{Ge23} for the (BO)--equation, where he found : for all $t\in\R\,,$ for almost every $x\in\R\,,$ and under the same condition of smoothness on the initial data of Theorem~\ref{Theorem solution mutivalué- intro}, the zero--dispersion limit of (BO) is given as
\be\label{zero dispersion limit BO}
    ZD_{(\mathrm{BO})}[u_0](t,x)=\sum_{k=0}^{2\ell}(-1)^k u_0(y_k^{BO})\,,
\ee
where the $(y_k^{BO})_{0,\,\ldots\,2\ell}\,,$ $\ell=\ell(x)\in\Nzero\,,$ are real solutions for the algebraic equation 
\[
y+2tu_0(y)=x\,.
\]
\end{Rq}

A consequence of the previous Theorem, is the existence of a maximum Principle.
\begin{corollary}[A Maximum Principle]
    Let $u_0\in L^2_+(\R)\cap L^\infty(\R)$ (with $\|u_0\|_{L^2}<\sqrt{2 \pi}$ in the focusing case). For all $t\in\R\,,$
    \[
    \|\zd\|_{L^\infty}\leq \|u_0\|_{L^\infty}\,.
    \]
\end{corollary}

\vskip0.3cm
\noindent
\textbf{Related works.} \textit{Zero dispersion limit of the KdV equation.}
The problem of zero dispersion limit was first investigated by Lax and Levermore \cite{LL83} for the KdV equation on the real line 
\be\tag{KdV}
    \p_tu-3\p_x(u^2)+\eps^2\p^3_x u=0\,,\qquad u^\eps(0,x)=u_0(x)\,,
\ee
describing, thus, the weak zero dispersion limit for nonpositive initial data decaying sufficiently fast at infinity. In contrast with the Benjamin--Ono equation \cite{Ge23} and the Calogero--Moser DNLS equation, the zero dispersion limit for the KdV equation is expressed implicitly, as it is characterized by a quadratic minimum problem with constraints. Lax--Levermore's work initiated a series of papers. One can cite \cite{Ve87,Ve91, GK07, CG09,CG10a,CG10b}\,, where in all these works the inverse scattering theory, the spectral theory of the Lax operator and the associated Riemann--Hilbert problem are the main keys.

 \textit{Zero dispersion limit of the Benjamin--Ono equation.}
We have previously referenced the research by \cite{Ge23}, which characterized the zero-dispersion limit of the Benjamin-Ono equation as an alternative sum  \eqref{zero dispersion limit BO}. However, this formula traces back to the work of \cite{MX11,MW16} and \cite{Ga23a,Ga23b} who had already derived this sum \eqref{zero dispersion limit BO} for specific examples of initial data and by using scattering theory or the spectral theory.

\vskip0.3cm
\noindent
\textbf{Acknowledgments.} The author would like to thank her Ph.D. advisor Patrick G\'erard for proposing this research problem and providing valuable comments on this paper.

\section{The explicit formula of the zero dispersion limit of (CM)}

This section aims to establish the existence of a weak (in $L^2$) zero--dispersion limit solution to the Calogero--Moser DNLS equation \eqref{CM}\,. Additionally, it seeks to properly characterize this weak limit for all time $t$\,, through an explicit formula.  However, before proceeding, it is necessary to revisit some properties regarding the \eqref{CM} equation.

\vskip0.25cm
It is a completely integrable PDE in the following two senses : 
First, it possesses a Lax Pair structure $(L_u,B_u)$ that satisfies the Lax equation 
\[
    \frac{dL_u}{dt}=[B_u\,,L_u]\,,\qquad\quad [B_u\,,L_u]:= B_uL_u -L_uB_u\;
\]
for enough regular $u$ satisfying the \eqref{CM}--equation \cite{GL22}. 
The Lax operators for \eqref{CM} are given by 
\be\label{Lu Bu}
    L_u=D\mp T_uT_{\bar{u}}\,,
	\quad \hskip0.9cm
	B_u=\pm T_{u}T_{\partial_x\bar{u}}\mp 
 T_{\partial_x{u}}T_{\bar{u}} +i(T_{u}T_{\bar{u}})^2\,,
\ee
where $T_v$ is the Toeplitz operator of symbol $v$ defined as 
\be\label{Tv}
    T_vf=\Pi(vf)\,, \qquad f\in\LtwoR\,,
\ee
and $\Pi$ denotes the Szeg\H{o} projector~\eqref{Szego proj}\,. 
Second,  the complete integrability manifests through the finding of an explicit formula of the solution of the \eqref{CM}--equation \cite{KLV23}\,. To introduce this formula, specific notation needs to be presented.
Thus, we consider on $L^2_+(\R)\,,$  the contraction semigroup 
\[
S(\eta)h(x)=\Pi(\eee^{ix\eta} h(x))\,, \qquad \eta>0\,.
\]
And we denote by $X$ its infinitesimal generator
\[
Xh(x)\,=\,-i\frac{d}{d\eta}_{\Big|{\eta=0}}(S(\eta)h(x))\,=\,xh(x)\,,
\]
of domain
\begin{align*}
    \operatorname{Dom}(X)
    =&\,\{h\in \LtwoR\,;\,xh \in L^2(\R)\}
    \\
    =&\,\{h\in \LtwoR\,;\,\fr{h} \in H^1\big([0,+\infty)\big)\,,\, \fr{h}(0)=0\}\,.
\end{align*}
Its adjoint $X^*$  has the following domain
\begin{align*}
    \operatorname{Dom}(X^*)
    =&\,\{f\in \LtwoR\,;\,\e \, c>0\,,\, \forall h \in \operatorname{Dom}(X) \,,\,\va{\ps{f}{Xh}}\leq c\,\|h\|_{L^2}\}
    \\
    =&\,\{f\in \LtwoR\,;\,\fr{f}\,|_{(0,+\infty)} \in H^1\big((0,+\infty)\big)\}\,,
\end{align*}
and is defined for all $\xi>0$ as 
\[
    \fr{X^*f}(\xi)=i\p_{\xi}\fr{f}(\xi)\,.
\]
That is, for all $f\in \operatorname{Dom}(X^*)\,,$ 
\begin{equation}\label{X*}
    X^*f(x)=xf+\frac{1}{2\pi i}
    \fr{f}(0^+)\,.
\end{equation}
The following theorem aims to recall the explicit formula of \eqref{CM} defined for any $u_0\in L^2_+(\R)$ \cite{KLV23}\,, and to extend the global well--posedness result in $H^k_+(\R):=H^k(\R)\cap L^2_+(\R)\,,$ $k\in\N_{\geq 1}$ obtained by \cite{GL22}\,, to $\LtwoR\,$ \cite{KLV23}\,.
\begin{theorem}[\cite{KLV23}] \label{GWP CM in L2 + Formule explicit}
    Let $u_0\in \LtwoR$ (such that $\|u_0\|_{L^2}<\sqrt{2\pi}$ in the focusing case). Then there exists a unique global solution $u\in \mathcal{C}_t L^2_+(\R)$ such that for any $(u_n^0)\subseteq H^\infty_+(\R)\,,$ $(xu_n^0)\subseteq L^2$\,, $u_n^0\to u_0$ in $L^2$\,, we have for all $T>0\,,$
    \[
    u_n \to u \qquad \text{in }\ \mathcal{C}_t L^2_+([-T,T]\,,\R)\,.
    \]
    Additionally, for all $z\in\C_+:=\lracc{z\in\C\,,\, \operatorname{Im}(z)>0}\,,$  
\begin{equation}\label{explicit formula CM}
    u(t,z)=\frac{1}{2\pi i}I_+((X^*+2tL_{u_0}-z)^{-1}u_0)\,,
\end{equation}
where $I_+$ denotes
\be\label{I+}
    I_+(f):=\fr{f}(0^+)\,, \qquad \qlq f\in \operatorname{Dom}(X^*)\,.
\ee
As a consequence, $\|u(t)\|_{L^2}=\|u_0\|_{L^2}\,.$
\end{theorem}

\begin{Rq}
    The explicit formula on $\Ltwo$ was earlier derived in \cite{Ba23a}[Proposition 3.4]. This is not the first instance of discovering an explicit formula for a completely integrable PDE; for previous examples, we refer to  \cite{GG15,Ge22,GP23}\,. 
\end{Rq}

\vskip0.25cm
In what follows, we consider the rescaled version of the  \eqref{CM}--equation with initial data $u_0\in L^2_+(\R)\cap L^\infty(\R)$ : For all $\eps>0\,,$
\be\label{CM-eps}\tag{CM--$\eps$}
    \begin{cases}
        i\p_t\ueps\,+\eps\,\p_x^2 \ueps \,\pm\,2D\Pi(|\ueps|^2)\ueps=0\,, 
        \\
        u^\eps|_{t=0}=u_0\,.
    \end{cases} 
\ee
Our primary focus is on establishing the existence of a weak zero dispersion limit for \eqref{CM}; that is, determining whether the \eqref{CM-eps} equation has a weak limit in $L^2$ as $\eps$ tends to 0. The following theorem addresses this question.
 We recall that the considered initial data $u_0$ is independent of the parameter $\eps$\,, and that $\zdp$ represents the weak zero-dispersion limit solution in the focusing case, and $\zdm$ corresponds to the one in the defocusing case.

\begin{theorem*}\ref{weak limit theorem zero dispersion}.
    Given  an initial data $u_0\in L^2_+(\R)\cap L^\infty(\R)\,,$  (with $\|u_0\|_{L^2}< \sqrt{2\pi}$ in the focusing case), the weak (in $L^2$--space) zero-dispersion limit solution $\zd$ of \eqref{CM} exists and is characterized via the following explicit formula
    \be\label{explicit formula R zero dispersion}
        \zd(t,z)=\Big(\operatorname{Id}\,\mp\,2tT_{u_0}T_{\bar{u}_0}(X^*-z)^{-1}\Big)^{-1}u_0(z)\,, \qquad t\in\R \,,\, z\in\C_+\,.
    \ee
    In addition, we have
    \be\label{zd can be extended to the Hardy space}
    \|\zd(t)\|_{L^2}\leq \|u_0\|_{L^2}\,.
    \ee
    Furthermore,
    if $u_0^n\to u_0$ strongly in $L^2$ as $n\to \infty$ with $\sup_n\|u_0^n\|_{L^\infty}<+\infty$\,,  then for all $T>0$,
    \be\label{cv faible du flot}
        \sup_{t\in[-T,T]}\va{ZD_{\pm}[u_0^n](t)- \zd(t)}\rightharpoonup 0 \ \text{ in } L^2(\R)\,.
    \ee
\end{theorem*}

\begin{proof}
    In view of  Theorem~\ref{GWP CM in L2 + Formule explicit}\, we have for all $\eps>0\,,$ $$\|u^\eps(t)\|_{L^2}=\|u_0\|_{L^2}\,,$$
    where $u^\eps(t)$ is the solution of \eqref{CM-eps}\,.
    Hence, by Banach’s theorem, we deduce that there exists $\zd\in\LtwoR$, such that, up to a sequence, $u^\eps(t)\rightharpoonup ZD[u_0](t)$ in $L^2$ as $\eps\to 0\,,$ and 
    \[
        \|\zd(t)\|_{L^2}\leq \liminf_{\eps\to0}\|u^\eps(t)\|_{L^2}=\|u_0\|_{L^2}\,.
    \]
    To characterize $\zd$\,, we will use the explicit formula of Theorem~\ref{GWP CM in L2 + Formule explicit}\,.
    However, first observe that when $u$ is a solution of \eqref{CM} for an initial data $u_0$\,, then  
$
    \sqrt{\eps}u(\eps t, \cdot)\equiv \sqrt{\eps} \mathcal{S}(\eps t)[u_0]
$ is a solution to \eqref{CM-eps} for an initial data $\sqrt{\eps} u_0\,,$ where $\mathcal{S}(t)$ denotes the flow of \eqref{CM}\,. 
That is $$\ueps(t,z):=\sqrt{\eps}\mathcal{S}(\eps t) \left[\frac{u_0}{\sqrt{\eps}}\right]$$ is a solution to \eqref{CM-eps} for an initial data $u_0\,.$
Therefore, starting from an initial data $u_0\,,$ one deduces by \eqref{explicit formula CM} and \eqref{Lu Bu} that the solution of \eqref{CM-eps} in the focusing and defocusing case is explicitly given, for all $\eps>0\,,$ by
\be\label{explicit formula rescaled eps}
    \ueps(t,z):=\frac{1}{2\pi i}I_+((X^*+2t\eps D\mp 2tT_{u_0}T_{\overline{u_0}}-z)^{-1}u_0)\,, \qquad z\in\C_+\,.
\ee
The next step is to pass to the limit $\eps\to 0\,$ in the above formula. For this purpose, we rewrite \eqref{explicit formula rescaled eps} as follows 
\[
    \ueps(t,z)=\frac{1}{2\pi i}\Big(\operatorname{Id}\,\mp\, 2t\mathrm{e}^{-i \eps t D^2}T_{u_0}T_{\bar{u}_0}\mathrm{e}^{i \eps t D^2}(X^*-z)^{-1}
    \Big)^{-1}u_0\,.
\]
Indeed, by using the Fourier transform, for all $\xi>0\,,$
\[
    \fr{(X^*+2t\eps D)f}(\xi)=\eee^{it\xi^2}i\p_\xi (\eee^{-it\xi^2} \fr{f} (\xi))\,,
\]
 \eqref{explicit formula rescaled eps} becomes
\begin{align*}
    \ueps(t,z)
    =&\,\frac{1}{2\pi i}I_+\Big((\eee^{i\eps t D^2}X^*\eee^{-i\eps t D^2}\mp 2tT_{u_0}T_{\overline{u_0}}-z)^{-1}u_0\Big)
\\
    =&\,\frac{1}{2\pi i}I_+\Big(\eee^{i\eps t D^2}(X^*\,\mp\, 2t\eee^{-i\eps t D^2}T_{u_0}T_{\overline{u_0}}\eee^{i\eps t D^2}-z)^{-1}\eee^{-i\eps t D^2}u_0\Big)
\end{align*}
Thus, by definition of $I_+$ in \eqref{I+}\,, we deduce
\begin{align*}
    \ueps(t,z)
    =\,&\frac{1}{2\pi i}I_+\Big((X^*\,\mp\, 2t\eee^{-i\eps t D^2}T_{u_0}T_{\overline{u_0}}\eee^{i\eps t D^2}-z)^{-1}\eee^{-i\eps t D^2}u_0\Big)
    \\
    =\,&\frac{1}{2\pi i}I_+\Big(\big(X^*-z\big)^{-1}\cdot\big(\Id\,\mp\, 2t\eee^{-i\eps t D^2}T_{u_0}T_{\overline{u_0}}\eee^{i\eps t D^2}(X^*-z)^{-1}\big)^{-1}\eee^{-i\eps t D^2}u_0\Big)
\end{align*}
Now, using the fact that 
\begin{align}\label{remove I+}
    I_{+}\left(\left(X^*-z\right)^{-1} f\right) & 
    =\, \lim _{\varepsilon \rightarrow 0}\left\langle\left(X^*-z\right)^{-1} f\,,\, \frac{1}{1-i\eps x}\right\rangle
    =\lim _{\varepsilon \rightarrow 0}\left\langle f\,,\,(X-\bar{z})^{-1} \left(\frac{1}{1-i\eps x}\right)\right\rangle\notag 
    \\
    & =\lim _{\varepsilon \rightarrow 0}\left\langle f, (x-z)^{-1}\left(\frac{1}{1-i\eps x}\right)\right\rangle=2 \pi i f(z) ,
\end{align}
we infer,
\be\label{u-eps sans I+}
     \ueps(t,z)
    =\Big(\Id\,\mp\, 2t\eee^{-i\eps t D^2}T_{u_0}T_{\overline{u_0}}\eee^{i\eps t D^2}(X^*-z)^{-1}\Big)^{-1}\eee^{-i\eps t D^2}u_0(z)\,.
\ee
Observing first that $\|\eee^{-i\eps t D^2}u_0\|_{L^2}= \|u_0\|_{L^2}\,,$ and second,   $\eee^{-i\eps t D^2}T_{u_0}T_{\overline{u_0}}\eee^{i\eps t D^2}(X^*-z)^{-1}$ is a bounded operator as $u_0\in L^\infty\,,$ we infer by passing to the limit $\eps\to 0$ in \eqref{u-eps sans I+}\,,
\[
    \zd(t,z):= \Big(\Id\,\mp\, 2tT_{u_0}T_{\overline{u_0}}(X^*-z)^{-1}\Big)^{-1}u_0(z)\,.
\]
\end{proof}

  \vskip0.5cm
\section{Link with the multivalued solution of the Burgers equation }
\vskip0.25cm
The aim of this section is to prove Theorem~\ref{Theorem solution mutivalué- intro}\,, which describes the weak zero dispersion limit solution of \eqref{CM} starting from an initial data $u_0\in L^2_+(\R)\cap \mathcal{C}^1$  tending to $0$ at infinity and satisfying $u_0'\in L^\infty(\R)$, in terms of the branches of the multivalued solution for the Burgers equation \eqref{Burgers equation}\,.  However, before proving this theorem for such initial data~$u_0$\,, we will first focus on proving it for rational initial data in the Hardy space
\be\label{u0 frac rationnel}
    u_0(y)=\frac{P(y)}{Q(y)}\,, \qquad Q(y):=(y+\overline{p}_0)\ldots(y+\overline{p}_{N-1})\,,\qquad p_k\neq p_j\,,\ k\neq j\,.
\ee
where $p_k\in\C$\,, $\operatorname{Im}(p_k)<0$ for all $k=0,\cdots\,, N-1$\,, and $P(y)=\sum_{n=0}^{N-1}a_ny^n$\,, $a_n\in\C\,$. 

\begin{prop}\label{Zd-solution multivalue prop}
    Let $u_0$ be a rational function defined in \eqref{u0 frac rationnel}\,. Then for every time $t\in\R\,,$ 
    and for almost every $x\in\R\,,$ the algebraic equation
    \be\label{eqt alg prop}
        y\mp2t\va{u_0(y)}^2=x
    \ee
    has an odd number of simple real solutions
    $
        y_0:=y_0(t,x)\,<\, \ldots\,<\, y_{2\ell}:=y_{2\ell}(t,x)\,,
    $
    and the zero--disperion limit of \eqref{CM} is given,
    for almost every $x\in \R$\,, 
    by
    \be \label{ZD sol multivalue for u0 rational function}
        \zd (t,x)=\mathrm{e}^{i \varphi(t,x)}\left(\mp i\, \frac{\va{t}}{t}\right)^\ell \prod_{k=0}^{2\ell}\va{u_0(y_k)}^{(-1)^k} 
\ee
where 
\[
    \varphi(t,x)=\arg(u_0(x))+\frac{1}{2\pi}\int_0^{+\infty}
    \frac{1}{s}\log\left(\frac{\ s\mp 2t \va{u_0(x+s)}^2}{-s\mp 2t \va{u_0(x-s)}^2}\frac{\prod_{k=0}^{2\ell}(x-s-y_k)}{\prod_{k=0}^{2\ell}(x+s-y_k)}\right)\mathrm{d}s\,.
\]
\end{prop}

\vskip0.3cm
To prove the previous proposition, we split the proof into the following lemmas.

\begin{lemma}\label{resolution eqt algebric}
    Let $u_0(y)=\frac{P(y)}{Q(y)}$ be a rational function defined as in \eqref{u0 frac rationnel}\,.
    Then, for all $t\in\R\,,$   $x\in\R\bk K_t\,,$ where $K_t$ is a finite set in $\R\,,$ the algebraic equation~\eqref{eqt alg prop}
    admits an odd number of simple real solutions
    $$
        y_0:=y_0(t,x)\,<\, \ldots\,<\, y_{2\ell}:=y_{2\ell}(t,x)\,.
    $$
    Furthermore, the function $\gamma_t(y):=y\mp 2t \va{u_0(y)}^2$ is increasing near $y_{2k}\,,$ $k=0\,,\cdots,\,\ell$\,, and decreasing near $y_{2k+1}\,,$ $k=0\,,\cdots,\,\ell-1$\,.
\end{lemma}
\begin{figure}[ht]
    \begin{tikzpicture}
    \begin{axis}[
        axis lines = center,
        xlabel = \(y\),
        ylabel = {\(x\)},
    ]
    \addplot [
        domain=-8:8, 
        samples=100, 
        color=red,
    ]
    {x-4/(x^2+1)};
    \addplot [
        domain=-8:8, 
        samples=100, 
        color=blue,
        ]
        {-3};
    \addplot [
        domain=-8:8, 
        samples=100, 
        color=green,
        ]
        {1}; 
    \end{axis}
    \end{tikzpicture}
    \caption{For an initial data $u_0(y):=\frac{1}{y+i}$\,, we have in red the graph of $\gamma_2(y):=y-2\cdot2\va{u_0(y)}^2=y-\frac{2\cdot2}{y^2+1}$\,. In green, we have the graph of $x=1\,.$ The abscissa of the intersection of the axe $x=1$ with the graph of $\gamma_2(y)$ corresponds to the unique real solution $y_0$ of $\gamma_2(y)=1\,.$ In blue, we have the graph of $x=-3\,.$ The abscissas of the intersection of the graph $\gamma_2(y)$ with $x=-3$ correspond to the points $y_0<y_1<y_2$ solutions to the algebraic equation $\gamma_2(y)=-3\,.$  }
\end{figure}
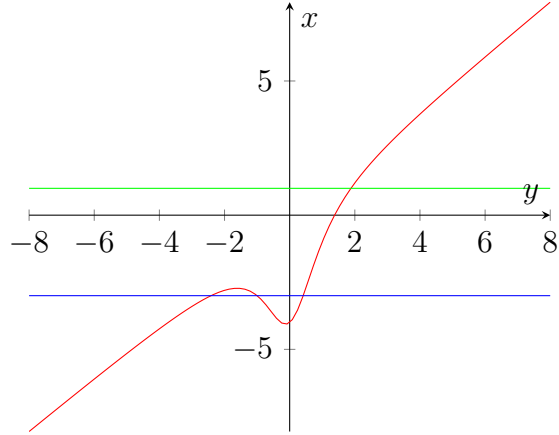
\begin{proof}
    Given $u_0(y)=\frac{P(y)}{Q(y)}$ as in \eqref{u0 frac rationnel}\,, we introduce~\footnote{~Observe that when $y\in\R\,,$ $v_0(y)=\va{u_0(y)}^2\,.$} 
    $$
    y\in\C\longmapsto v_0(y):=\frac{P(y)\overline{P}(y)}{Q(y)\overline{Q}(y)}\,.
    $$
    Our goal is to study the \textit{real} solutions of the algebraic equation~\eqref{eqt alg prop}\,. Observe  that, $y$ is a real solution of~\eqref{eqt alg prop}\,, if and only if,
    $y$ is a real solution of 
    $$
        y\mp 2t v_0(y)=x\,, \qquad t\in\R\,,
    $$
    that is, if and only if, $y$ is a real solution of
    the polynomial equation of degree $2N+1$\,, 
    \be\label{eqt pol}
        (y-x)Q(y)\overline{Q}(y)\,\mp\,2t P(y)\bar{P}(y)=0\,.
    \ee
Now, focusing on \eqref{eqt pol}, one notices $y\in\C\bk\R$ is a solution to  \eqref{eqt pol}\,, if and only if, its complex conjugate $\bar{y}$ is a solution to  \eqref{eqt pol}\,.
    Therefore, the polynomial equation~\eqref{eqt pol} 
     of degree $2N+1$ 
    admits an odd number of real solutions 
    $$
        y_0:=y_0(t,x)\,\leq\, \ldots\,\leq\, y_{2\ell}:=y_{2\ell}(t,x)\,.
    $$
    Discarding a finite set of critical values\footnote{~We mean by critical values of a function $\gamma_t$, the $\gamma_t$--image of the critical points of $\gamma_t$\,, i.e. the $\gamma_t$--images of the points where $\gamma_t'(y)=0\,.$} $x$ of the function $\gamma_t(y):=y\mp 2t \va{u_0(y)}^2$  for a given $t$\,, one may assume that these real solutions $y_k$ are simple and that the function 
   $\gamma_t$
    is increasing in a neighborhood of the points $y_{2k}\,,$ $k=0, \ldots\,,\ell\,,$ and decreasing in a neighbor of the points $y_{2k+1}$\,, $k=0\,,\ldots\,,,\ell-1\,,$ as $y\longmapsto \gamma_t(y)$ is a continuous function behaving like
    \bes
        \gamma_t(y)\underset{y\to\pm\infty}{\sim} y\,. 
    \ees
\end{proof}

\begin{lemma}\label{ZD=Determinant}
Let $u_0(y)=\frac{P(y)}{Q(y)}$ be the rational function  and $(p_k)_{k=0\,,\,\cdots\,,\,N-1}$ be the complex constants defined in \eqref{u0 frac rationnel}\,.   Moreover denote, for almost every $x\in\R\,,$  $y_0:=y_0(t,x)\,,\,\cdots\,,\,y_{2N}:=y_{2N}(t,x)$  the solutions of the equation
    \be\label{eqt alg P et Q}
        y\mp 2t\,\frac{P(y)\overline{P}(y)}{Q(y)\overline{Q}(y)}=x\,, \qquad t\in\R\,.
    \ee
   Then, the zero--dispersion limit of \eqref{CM} is given by 
\be\label{ZD u0= det/det-lemme}
           \zd(t,x)=
           \frac{u_0(y_0) u_0(y_2)\ldots u_0(y_{2N})
           \left|\begin{array}{ccccc}
            1&\frac{1}{y_0+p_0} & \frac{1}{y_0+p_1} & \ldots & \frac{1}{y_0+p_{N-1}}
            \\
            1 & \frac{1}{y_2+p_0} & \frac{1}{y_2+p_1} & \ldots & \frac{1}{y_2+p_{N-1}}
            \\
            \vdots
            \\
             1& \frac{1}{y_{2N}+p_0}&\frac{1}{y_{2N}+p_1} & \ldots & \frac{1}{y_{2N}+p_{N-1}}
            \end{array}\right|}{\left|\begin{array}{ccccc}
            1&\frac{u_0(y_0)}{y_0+p_0} & \frac{u_0(y_0)}{y_0+p_1} & \ldots & \frac{u_0(y_0)}{y_0+p_{N-1}}
            \\
            1 & \frac{u_0(y_2)}{y_2+p_0} & \frac{u_0(y_2)}{y_2+p_1} & \ldots & \frac{u_0(y_2)}{y_2+p_{N-1}}
            \\
            \vdots
            \\
             1& \frac{u_0(y_{2N})}{y_{2N}+p_0}&\frac{u_0(y_{2N})}{y_{2N}+p_1} & \ldots & \frac{u_0(y_{2N})}{y_{2N}+p_{N-1}}
            \end{array}\right|}\,.
    \ee
\end{lemma}

\begin{proof}
The main component is to use the explicit formula~\eqref{explicit formula R zero dispersion-intro} of $\zd$\,, 
which can be 
be reexpressed using~\eqref{remove I+} as 
\be\label{zd[u0] preuve}
    \zd(t,z)=\frac{1}{2\pi i}I_+\big[(X^*\mp 2tT_{u_0}T_{\overline{u_0}}-z)^{-1}u_0\big]\,, \qquad z\in\C_+\,,
   \ee
   where $I_+$ is defined in \eqref{I+}\,.   The goal is to transform \eqref{zd[u0] preuve} into \eqref{ZD u0= det/det-lemme}\,. For that, we decompose $u_0$ in terms of its
partial fractional decomposition $u_0(y)=\sum_{k=0}^{N-1}\frac{c_k}{y+\overline{p_k}}\,,$ $c_k\in\C\,,$
to infer by \eqref{Tv} and \eqref{X*}\,,
      \begin{align}\label{X*+2tTuTu-z}
       (X^*\,\mp\,2tT_{u_0}T_{\overline{u_0}}\,-z)f(y)
       =\,(y\,\mp\,
        2t\va{u_0(y)}^2-z)f(y)\,+\,\frac{1}{2\pi i}I_+(f)\,\;\notag
        \\
        \,\pm\, 2t u_0(y) \sum_{k\geq 0}^{N-1}\frac{\overline{c_k}}{y+p_k}f(-p_k)\,.
   \end{align}
   Indeed, for all $f\in\LtwoR\,,$
    $$
        T_{\overline{u_0}}f(y)=\sum_{k=0}^{N-1}\,\Pi_+\left(\frac{\overline{c_k}}{y+p_k}f(y)\right)
        =\overline{u_0}(y)f(y)-\sum_{k=0}^{N-1}\frac{\overline{c_k}}{y+p_k}f(-p_k)\,.
    $$
   Thus, for all $f\in\LtwoR\,,$
   \[
        T_{u_0}T_{\overline{u_0}}f(y)=\va{u_0(y)}^2f(y)-u_0(y)\sum_{k\geq 0}^{N-1}\frac{\overline{c_k}}{y+p_k}f(-p_k)\,.
   \]
Now, observe since formula~\eqref{X*+2tTuTu-z} is valide for any $f\in\LtwoR\,,$ then one can extended it to any holomorphic function $f$ in $\C_+$ whose trace on $\R$ is in $L^2(\R)\,.$ That is, if  we denote by 
\bes
    y\in\C\longmapsto v_0(y):=\frac{P(y)\bar{P}(y)}{Q(y)\overline{Q}(y)}\,,
\ees
then the following identity holds
   \begin{align}\label{X*+2tTuTu-z-complex}
       (X^*\,\mp\,2tT_{u_0}T_{\overline{u_0}}\,-z)f(y)
       =\,(y\,\mp\,
        2tv_0(y)-z)f(y)\,+\,\frac{1}{2\pi i}I_+(f)\;\notag
        \\
        \,\pm\, 2t u_0(y) \sum_{k\geq 0}^{N-1}\frac{\overline{c_k}}{y+p_k}f(-p_k)\,,
   \end{align}
   for all $y\in\C_+$\,, and for all holomorphic function $f$ on $\C_+$ whose trace is in $L^2\,.$
   In particular, for  $f(y)=f_{t,z}(y):=(X\,\mp\,2tT_{u_0}T_{\overline{u_0}}\,-z)^{-1}u_0(y)\in \LtwoR$\,, we  infer by~\eqref{zd[u0] preuve}\,,
   \[
        u_0(y)=(y\,\mp\,2tv_0(y)-z)f_{t,z}(y)+\zd(t,z)\,\pm\, 2t u_0(y) \sum_{k\geq 0}^{N-1}\frac{\overline{c_k}}{y+p_k}f_{t,z}(-p_k)\,,
   \]
    or
    \be\label{f[tz]}
    f_{t,z}(y)=\frac{\,u_0(y)-\zd(t,z)\,\mp\, 2t u_0(y)  \displaystyle\sum_{k\geq 0}^{N-1}\frac{\overline{c_k}}{y+p_k}f_{t,z}(-p_k)\,}{y\,\mp\,2tv_0(y)-z}\,.
    \ee
    However, recall  $y\mapsto f_{t,z}(y)$
   is a holomorphic function in the upper--half complex plane. This means that the zeros in $\C_+$ of the denominator of $f_{t,z}$ must cancel its numerator.
    Therefore, the next step is to find the zeros of the algebraic equation $y\mp2t v_0(y)=z$ on $\C_+$\,, with the note that, at the end of the day, $z\in\C_+$ will be replaced by $x\in\R$ almost everywhere, as $\zd$ belongs to the Hardy space $L^2_+(\R)$ thanks to \eqref{zd can be extended to the Hardy space}\,.

\vskip0.25cm
Let $x\in\R\,.$  In view of Lemma~\ref{resolution eqt algebric}\,, the algebraic equation $y\mp2t v_0(y)=x $ admits an odd number of real solutions~\footnote{~We recall that the real solutions $y$ of the equation  $y\mp2t v_0(y)=x $
 are the same real solutions of $y\mp2t \va{u_0(y)}^2=x $ as $v_0(y)=\va{u_0(y)}^2$ when $y$ is real.}
  $$
        y_0:=y_0(t,x)\,<\, \ldots\,<\, y_{2\ell}:=y_{2\ell}(t,x)\,.
    $$
Moreover, we denote by 
    \[
        y_{2\ell+1}:=y_{2\ell+1}(t,x)\,,\,\ldots\,,\, y_{2N}:=y_{2N}(t,x)\,,
    \]
    the remaining solutions of $y\mp2t v_0(y)=x $ 
    belonging to the complex plane,  where thanks to \eqref{eqt pol} we notice $y_{2p-1}=\overline{y_{2p}}$\, for all $p=\ell+1\,,\,\cdots\,,\,N\,$; and in what follows, we suppose $\operatorname{Im}(y_{2p})>0$ for all $p=\ell+1\,,\,\cdots\,,\,N\,.$ 
    \vskip0.2cm
     By moving $x= z$ slightly up to the upper half--complex plane, one proves by using the implicit function theorem in its holomorphic version, that $ z\mapsto y_k(t,z)$ is a holomorphic function and thus satisfies the Cauchy--Riemann equations
    \bes
        \frac{\p\operatorname{Im}(y_k)}{\p\operatorname{Im}(z)}=\frac{\p\operatorname{Re}(y_k)}{\p\operatorname{Re}(z)}\,.
    \ees
    Besides, recall for all $k=0\,, \ldots \ell\,,$ $j=1\,,\ldots\,,\ell \,,$ 
    \bes
        \frac{\p\operatorname{Re}(y_{2k})}{\p\operatorname{Re}(z)}>0\,,
        \qquad
        \frac{\p\operatorname{Re}(y_{2j-1})}{\p\operatorname{Re}(z)}<0\,,
    \ees
since by Lemma~\ref{resolution eqt algebric} the function $\gamma_t(y):=y\mp 2t v_0(y)$ is increasing near $y_{2k}\,,$ $k=0\,,\cdots,\,\ell$\,, and decreasing near $y_{2k+1}\,,$ $k=0\,,\cdots,\,\ell-1$\,.
     As a result,
    for all $k=0\,,\ldots N\,,$ $j=1\,,\ldots\,,N\,,$
    \[
        \frac{\p\operatorname{Im}(y_{2k})}{\p\operatorname{Im}(z)}>0\,,
        \qquad
        \frac{\p\operatorname{Im}(y_{2j-1})}{\p\operatorname{Im}(z)}<0\,.
    \]
     That is $(y_{2k})_{k=0\,,\ldots N}\subseteq\C_+$\,, and thus, at these points, the numerator of \eqref{f[tz]} must vanish.
    Consequently,   one deduces the following linear system of unknowns $ZD[u_0](t,z)\,,$ and $(f_{t,z}(-p_j))_{j=1,\ldots, N-1}$\,,
    \vskip0.001cm
    $$
        u_0(y_{2k})=\zd(t,z)\,\pm\, 2t u_0(y_{2k}) \sum_{j\geq 0}^{N-1}\frac{\overline{c_j}}{y_{2k}+p_j}f_{t,z}(-p_j)\,, \qquad k=0, \cdots,N\,.
    $$
    \vskip0.1cm
    \noindent
     Applying the Cramer rule, one finds for all $z\in\C_+\,,$ 
     \vskip0.05cm
      \be\label{ZD u0= det/det}
           \zd(t,z)=
           \frac{u_0(y_0) u_0(y_2)\ldots u_0(y_{2N})
           \left|\begin{array}{ccccc}
            1&\frac{1}{y_0+p_0} & \frac{1}{y_0+p_1} & \ldots & \frac{1}{y_0+p_{N-1}}
            \\
            1 & \frac{1}{y_2+p_0} & \frac{1}{y_2+p_1} & \ldots & \frac{1}{y_2+p_{N-1}}
            \\
            \vdots
            \\
             1& \frac{1}{y_{2N}+p_0}&\frac{1}{y_{2N}+p_1} & \ldots & \frac{1}{y_{2N}+p_{N-1}}
            \end{array}\right|}{\left|\begin{array}{ccccc}
            1&\frac{u_0(y_0)}{y_0+p_0} & \frac{u_0(y_0)}{y_0+p_1} & \ldots & \frac{u_0(y_0)}{y_0+p_{N-1}}
            \\
            1 & \frac{u_0(y_2)}{y_2+p_0} & \frac{u_0(y_2)}{y_2+p_1} & \ldots & \frac{u_0(y_2)}{y_2+p_{N-1}}
            \\
            \vdots
            \\
             1& \frac{u_0(y_{2N})}{y_{2N}+p_0}&\frac{u_0(y_{2N})}{y_{2N}+p_1} & \ldots & \frac{u_0(y_{2N})}{y_{2N}+p_{N-1}}
            \end{array}\right|}\,.
    \ee
    \vskip0.1cm
    Hence, for almost every $x\in\R\,,$ identity  \eqref{ZD u0= det/det-lemme} holds, as 
    $\zd(t)$ belongs to the Hardy space $L^2_+(\T)$ for all $t\,,$ since it is a holomorphic function, exhibiting a finite trace in $L^2(\R)$ thanks to inequality~\eqref{zd can be extended to the Hardy space}\,. 
     
\end{proof}

\begin{lemma}[Solving the determinants of Lemma~\ref{ZD=Determinant}]\label{Solving the determinant}
Under the same conditions and notations as in Lemma~\ref{ZD=Determinant}\,,
   the zero--dispersion limit of \eqref{CM} associated with $u_0=\frac{P(y)}{Q(y)} $ defined in \eqref{u0 frac rationnel}\,, is given for almost every $x\in\R$ by
    \begin{align}\label{resolution system}
        \zd(t,x)
            \;=\;
            \frac{P(x)}{\displaystyle\prod_{k=1}^N(x-y_{2k-1})}\;.
    \end{align}
\end{lemma}

\begin{proof}
We recall from Lemma~\ref{ZD=Determinant}\,, for almost every $x\in\R\,,$
\[
\zd(t,x)=
           \frac{u_0(y_0) u_0(y_2)\ldots u_0(y_{2N})
           \left|\begin{array}{ccccc}
            1&\frac{1}{y_0+p_0} & \frac{1}{y_0+p_1} & \ldots & \frac{1}{y_0+p_{N-1}}
            \\
            1 & \frac{1}{y_2+p_0} & \frac{1}{y_2+p_1} & \ldots & \frac{1}{y_2+p_{N-1}}
            \\
            \vdots
            \\
             1& \frac{1}{y_{2N}+p_0}&\frac{1}{y_{2N}+p_1} & \ldots & \frac{1}{y_{2N}+p_{N-1}}
            \end{array}\right|}{\left|\begin{array}{ccccc}
            1&\frac{u_0(y_0)}{y_0+p_0} & \frac{u_0(y_0)}{y_0+p_1} & \ldots & \frac{u_0(y_0)}{y_0+p_{N-1}}
            \\
            1 & \frac{u_0(y_2)}{y_2+p_0} & \frac{u_0(y_2)}{y_2+p_1} & \ldots & \frac{u_0(y_2)}{y_2+p_{N-1}}
            \\
            \vdots
            \\
             1& \frac{u_0(y_{2N})}{y_{2N}+p_0}&\frac{u_0(y_{2N})}{y_{2N}+p_1} & \ldots & \frac{u_0(y_{2N})}{y_{2N}+p_{N-1}}
            \end{array}\right|}\,.
\]
    By expanding the determinants, one finds
    \[
    \zd(t,x)=\prod_{k=0}^{2N}u_0(y_{2k}) \;
        \frac{\displaystyle\sum_{k=0}^N(-1)^k\,\Delta_k}{\displaystyle\sum_{k=0}^N(-1)^k\,\frac{1}{u_0(y_{2k})}\,\Delta_k}
    \]
    where $\Delta_k$ is the minor obtained after removing the first column and the $k^{\text{th}}$--row of the matrix in the numerator of \eqref{resolution system}\,.
    Therefore, observing that $\Delta_k$ is a Cauchy determinant, one infers that 
    \be\label{ZDu=R/D}
\zd(t,x)=\frac{\displaystyle\sum_{k=0}^N(-1)^k\;\displaystyle\prod_{j=0}^{N-1}(y_{2k}+p_{j})\,\prod_{\underset{j,j'\neq k}{0\leq j<j'\leq N}}(y_{2j'}-y_{2j})}{
           \displaystyle\;\sum_{k=0}^{N}(-1)^k\;\frac{1}{u_0(y_{2k})}\displaystyle\displaystyle \prod_{j=0}^{N-1}(y_{2k}+p_{j})\prod_{\underset{s,s'\neq k}{0\leq s<s'\leq N}}(y_{2s'}-y_{2s})}=:\,
           \frac{R}{D}
           \,.\
\ee
First,  by using the Vandermonde determinant, one can rewrite $R$ in \eqref{ZDu=R/D}  as
\begin{align*}
        R:= &\;\sum_{k=0}^N(-1)^k\,\displaystyle\,\prod_{j=0}^{N-1}(y_{2k}+p_{j})\, \prod_{\underset{j,j'\neq k}{0\leq j<j'\leq N}}(y_{2j'}-y_{2j})
        \\
        =&\;\left|\begin{array}{ccccc}
           \displaystyle \prod_{j=0}^{N-1}(y_{0}+p_{j})& 1 & y_0 & \ldots & y_0^{N-1}
            \\
            \displaystyle \prod_{j=0}^{N-1}(y_{2}+p_{j}) & 1 & y_2 & \ldots & y_2^{N-1}
            \\
            \vdots
            \\
            \displaystyle \prod_{j=0}^{N-1}(y_{2N}+p_{j})& 1&y_{2N} & \ldots & y_{2N}^{N-1}
            \end{array}\right|
    \end{align*}
    which is equivalent to
    \[
        R=\left|\begin{array}{ccccc}
           y_0^N& 1 & y_0 & \ldots & y_0^{N-1}
            \\
             y_2^N & 1 & y_2 & \ldots & y_2^{N-1}
            \\
            \vdots
            \\
             y_{N-1}^N& 1&y_{2N} & \ldots & y_{2N}^{N-1}
            \end{array}\right|
    \]
thanks to the property of the determinant. Hence, applying once more the Vandermonde determinant, we conclude that 
    \be\label{R}
        R=(-1)^N\prod_{0\leq m<n\leq N}(y_{2n}-y_{2m})\,.
    \ee
Second, moving to the expression of $D$ in \eqref{ZDu=R/D}\,. Observe that by definition of the polynomial $Q$ in~\eqref{u0 frac rationnel}\,, we have 
\be\label{remplacer Q} 
    \frac{1}{u_0(y_{2k})}\prod_{j=0}^{N-1}(y_{2k}+p_j)\equiv\frac{\overline{Q}(y_{2k})}{u_0(y_{2k})}\,.
\ee
In addition,  recall that the $(y_{2k})_{k=0\,,\ldots N}$ are solutions of the algebraic equation~\eqref{eqt pol}\,. Hence, 
for all $k=0\,,\cdots\,, N$\,, 
    \[
   \frac{\overline{Q}(y_{2k})}{u_0(y_{2k})}\,=\,\mp 2t \frac{ \overline{P}(y_{2k})}{x-y_{2k}}\,,
    \]
    which can be rewritten as
    \[
      \frac{\overline{Q}(y_{2k})}{u_0(y_{2k})}\,
    =\,\mp 2t\left(\frac{\overline{P}(y_{2k})-\overline{P}(x)}{x-y_{2k}}+\frac{\overline{P}(x)}{x-y_{2k}}\right)\,,
    \]
    to infer via \eqref{remplacer Q} 
    \be\label{1/u0=P}
        \frac{1}{u_0(y_{2k})}\prod_{j=0}^{N-1}(y_{2k}+p_j)=\mp 2t\left(\frac{\overline{P}(y_{2k})-\overline{P}(x)}{x-y_{2k}}+\frac{\overline{P}(x)}{x-y_{2k}}\right)\,.
    \ee
Therefore, by observing that $\frac{\overline{P}(y_{2k})-\overline{P}(x)}{x-y_{2k}}$ is a polynomial in $y_{2k}$ of degree less strictly than $N-1\,,$ one finds by~\eqref{1/u0=P}\,, that $D$ defined in \eqref{ZDu=R/D} is equal to, 
\begin{align*}
        D
        =&\;\mp 2t\, \overline{P}(x)\left|\begin{array}{ccccc}
           \displaystyle \frac{1}{x-y_0}& 1 & y_0 & \ldots & y_0^{N-1}
            \\[0.3cm]
            \displaystyle \frac{1}{x-y_2} & 1 & y_2 & \ldots & y_2^{N-1}
            \\
            \vdots
            \\
            \displaystyle \frac{1}{x-y_{2N}}& 1&y_{2N} & \ldots & y_{2N}^{N-1}
            \end{array}\right|\,.
    \end{align*}
    Consequently, by using the Vandermonde determinant,
    \[
    D= \,\mp 2t \overline{P}(x)\,\sum_{k=0}^N\frac{(-1)^k}{x-y_{2k}}\prod_{\underset{m,n\neq k}{0\leq m<n\leq N}}(y_{2n}-y_{2m})\,,
    \]
    which can be rewritten as
    \be\label{D}
    D= \mp 2t \, (-1)^N \,\overline{P}(x)  \;  \frac{\displaystyle\prod_{0\leq m<n\leq N}(y_{2n}-y_{2m})}{\displaystyle\prod_{k=0}^N(x-y_{2k})}\,,
\ee
thanks to the partial fractional decomposition of
    \begin{align*}
        \frac{\displaystyle\prod_{0\leq m<n\leq N}(y_{2n}-y_{2m})}{\displaystyle\prod_{k=0}^N(x-y_{2k})}
        =&\,\sum_{k=0}^N\;\frac{\frac{\displaystyle\prod_{0\leq m<n\leq N}(y_{2n}-y_{2m})}{\displaystyle\prod_{j=0\,, j\neq k}^N(y_{2k}-y_{2j})}\ \ }{ x-y_{2k}}
        \\
        =&\,
        \sum_{k=0}^N \;\frac{\displaystyle\frac{\displaystyle\prod_{0\leq m<n\leq N}(y_{2n}-y_{2m})}{\displaystyle (-1)^{N-k}\;\prod_{0\leq j<k}(y_{2k}-y_{2j})\prod_{k<j\leq N}(y_{2j}-y_{2k})}\ }{x-y_{2k}}
        \\[0.2cm]
        =&\,
        \sum_{k=0}^N\, \frac{(-1)^{N-k}}{x-y_{2k}}\prod_{\underset{m,n\neq k}{0\leq m<n\leq N}}(y_{2n}-y_{2m})\,.
    \end{align*}
Consequently, substituting \eqref{R}\,, \eqref{D} in \eqref{ZDu=R/D}\,, one infers that 
\be\label{x-y/P}
    \zd(t,x)=\frac{\displaystyle\prod_{k=0}^N(x-y_{2k})}{\mp 2t \overline{P}(x)}\;.
\ee
Furthermore, if we take into account that $(y_k)_{k=0\,,\,\cdots\,,\, 2N}$ are solutions to \eqref{eqt alg P et Q}\,,  thus also to the polynomial equation~\eqref{eqt pol}\,, we can write
\be\label{prod (y-yk)=P,Q}
   \prod_{k=0}^{2N}(y-y_k)=(y-x)Q(y)\overline{Q}(y)\mp 2t P(y)\overline{P}(y)\,,
\ee
and when $y=x$ in the above equation, we obtain $\displaystyle\prod_{k=0}^{2N}(x-y_k)=\mp 2t P(x)\overline{P}(x)\,, $ which implies that \eqref{x-y/P} can be replaced by
\be\label{ZD[u0] P(x)/prod}
    \zd(t,x)=\frac{P(x)}{\prod_{k=1}^N(x-y_{2k-1})}\,.
\ee

\end{proof}

Now, equipped with Lemma~\ref{Solving the determinant}\,, let us prove Proposition~\ref{Zd-solution multivalue prop}.

\begin{proof}[Proof of Proposition~\ref{Zd-solution multivalue prop}]
We recall from Lemma~\ref{Solving the determinant}\,, that for almost every $x\in\R\,,$
\be\label{Zd=P/prod}
    \zd(t,x)=\frac{P(x)}{\prod_{k=1}^N(x-y_{2k-1})}\,,
\ee
where the $(y_k)_{k=0\,,\,\cdots\,,\,2\ell}$ are the real solutions of the algebraic equation~\eqref{eqt alg prop}\,, and the $(y_p)_{p=2\ell+1\,,\,\cdots\,,\,2N}$ are the complex solutions of \eqref{eqt alg P et Q} with $y_{2p-1}=\overline{y_{2p}}$\, for all $p=\ell+1\,,\,\cdots\,,\,N\,.$ We rewrite \eqref{Zd=P/prod} as
\begin{align}\label{ZD-Q/prod}
    \zd(t,x)=&\,u_0(x)\,\frac{Q(x)}{\prod_{k=1}^N(x-y_{2k-1})}
    \\
    =&\, \frac{u_0(x)}{\prod_{k=1}^\ell(x-y_{2k-1})}\,\frac{Q(x)}{\prod_{p=\ell+1}^N(x-y_{2p-1})}\notag\,.
\end{align}
The goal is to get rid of $\frac{Q(x)}{\prod_{k=\ell+1}^N(x-y_{2k-1})}\,,$ in order to express $\zd$ only in terms of $y_0\,,\,\cdots\,,\, y_{2\ell}$\,, thereby ensuring that $\zd$ can be expressed exclusively in terms of the branches of the multivalued solution of the Burgers equation~\eqref{Burgers equation}\,.

\vskip0.2cm
\noindent
For that, we recall from~\eqref{prod (y-yk)=P,Q}\,, for all $y\in\C\,,$
\[
\frac{\displaystyle
\prod_{k=0}^{2N}(y-y_k)}{Q(y)\overline{Q}(y)}=y-x\,\mp\, 2t\frac{ P(y)\overline{P}(y)}{Q(y)\overline{Q}(y)}\,.
\]
In particular, for all $y\in\R$\,,
\[
    \frac{y \mp 2t \va{u_0(y)}^2-x}{\displaystyle\prod_{k=0}^{2\ell}(y-y_k)}
    \;=\;\frac{\displaystyle\prod_{p=2\ell+1}^{2N}(y-y_p)}{\va{Q(y)}^2}\,,
\]
or, since $y_{2p-1}=\overline{y_{2p}}$ for all $p=\ell+1\,,\ldots\,,N\,,$ and as $Q(y):=(y+\overline{p}_0)\cdots(y+\overline{p}_{N-1})$ by~\eqref{u0 frac rationnel}\,,
\be\label{g(y)}
    \frac{y \mp 2t \va{u_0(y)}^2-x}{\displaystyle\prod_{k=0}^{2\ell}(y-y_k)}
    \;=\;\frac{\displaystyle\prod_{p=\ell+1}^N\va{y-y_{2p-1}}^2}{\displaystyle\prod_{j=0}^{N-1}\va{y-p_j}^2}\;.
\ee
Let $a>0\,,$ the next step is to prove that the term we need to get ride of is equal to 
\be\label{Q/prod}
    \frac{Q(x)}{\displaystyle\prod_{p=\ell+1}^N(x-y_{2p-1})}
    =
    \frac{(x+ia)^\ell}{\exp\bigg(    \Pi_+\Big(\log \big((y^2+a^2)^\ell g_{t,x}(y)\big) \Big)\bigg)\Big|_{y=x}}\,,
\ee
where 
\bes
    g_{t,x}(y):=\;\frac{\displaystyle\prod_{p=\ell+1}^N\va{y-y_{2p-1}}^2}{\displaystyle\prod_{j=0}^{N-1}\va{y-p_j}^2}\,.
\ees
Indeed, by definition of $g_{t,x}\,,$ \footnote{~ The multiplication of $g_{t,x}$ by $(y^2+a^2)^\ell$ aims to ensure that each term on the right-hand side of the following identity is in $L^2(\R)$}
\[
    \log \Big((y^2+a^2)^\ell g_{t,x}(y)\Big)=\log\left(\frac{(y+ia)^\ell\prod_{p=\ell+1}^N(y-y_{2p-1})}{\prod_{j=0}^{N-1}(y-p_j)}\right)
    +
    \log\left(\frac{(y-ia)^\ell\prod_{p=\ell+1}^N(y-\overline{y_{2p-1}})}{\prod_{j=0}^{N-1}(y-\overline{p_j})}\right)\,,
\]
where one observes the first term of the right--hand side belongs to $L^2_+(\R)$\ as $\operatorname{Im}(p_k)<0\,$, while the second term belongs to $L^2_-(\R)$ \footnote{i.e. the space of functions having a trace in $L^2(\R)$\,, such that they can be holomorphically extended to $\C_-$)}\,.
 Therefore, by the uniqueness of the decomposition of any $L^2$--function in  $L^2_+(\R)\oplus L^2_-(\R)$\,, we infer for a fixed $x\in\R\,,$
\[
    \Pi_+\Big(\log \big((y^2+a^2)^\ell g_{t,x}(y)\big) \Big)=\log\left(\frac{(y+ia)^\ell\,\prod_{k=\ell+1}^N(y-y_{2k-1})}{Q(y)}\right)\,.
\]
Applying the exponential function to both sides of the previous identity and setting  $y=x$\,, one deduces \eqref{Q/prod}\,.
Now, the only task left is to compute the right--hand side of~\eqref{Q/prod}\,. To do so, we need the following classical lemma, the proof of which will be presented later for the convenience of the reader,
 \begin{lemma}\label{Pi+}
      For any $h\in L^2(\R)$ of class $\mathcal{C}^1$ satisfying $\|h'\|_{L^\infty}<\infty$\,,
      \[
            \Pi h(x)=\frac{h(x)}{2}-\frac{i}{2\pi}\int_0^{+\infty}\frac{h(x+s)-h(x-s)}{s}\,\mathrm{d}s\,,\qquad x\in\R\,.
      \]
  \end{lemma}
  Thus, applying this lemma with $h(y)=\log \big((y^2+a^2)^\ell g(y)\big)\,,$ one obtains
  \begin{align}\label{Pi(log((x2+a2)g(x)))}
    \Pi_+\Big(\log \big((y^2+a^2)^\ell g_{t,x}(y)\big) \Big)
    =&\,\frac{1}{2}\log\Big((y^2+a^2)^\ell g_{t,x}(y)\Big)
    -\frac{i\, \ell}{2\pi}\int_0^{+\infty}\frac{1}{s}\log\left(\frac{(y+s)^2+a^2}{(y-s)^2+a^2}\right) \mathrm{d}s\notag
    \\
    -&\,\frac{i}{2\pi}\int_0^{+\infty}
    \frac{\log(g_{t,x}(y+s))-\log(g_{t,x}(y-s))}{s}\,\mathrm{d}s\,.   
\end{align}
where after some computation, one finds
\be\label{2 pi arctan}
   \frac{1}{2\pi} \int_0^{+\infty}\frac{1}{s}\log\left(\frac{(y+s)^2+a^2}{(y-s)^2+a^2}\right)\, \mathrm{d}s= \frac{\pi}{2}-\arctan\left(\frac{a}{y}\right)\;.
\ee
\noindent
Indeed, let $f(y):=\int_0^{+\infty}\frac{1}{s}\log\left(\frac{(y+s)^2+a^2}{(y-s)^2+a^2}\right)\,\mathrm{d}s\,.$ By deriving $f$\,, one finds
\begin{align*}
    f'(y)
    =&\,\int_0^{+\infty}\frac{4(a^2+s^2-x^2)}{\big((y+s)^2+a^2\big)\big((y-s)^2+a^2\big)}\,\mathrm{d}s
    \\
    =&\,\int_{-\infty}^{+\infty}\frac{2(a^2+s^2-x^2)}{\big((y+s)^2+a^2\big)\big((y-s)^2+a^2\big)}\,\mathrm{d}s
    \\
    =&\,\lim_{R\to \infty}\int_{\Gamma_R} \frac{2(a^2+z^2-y^2)}{\big((y+z)^2+a^2\big)\big((y-z)^2+a^2\big)}\,\mathrm{d}z
\end{align*}
where $\Gamma_R$ is the contour composed of the real axis from $-R$ to $R$ and the upper semicircle.
Using Cauchy's residue theorem, we infer $f'(y)=\frac{2\pi a}{y^2+a^2}\,.$ Therefore, by integrating the previous expression and observing that $f(0)=0\,,$ and using that $\arctan(\theta)+\arctan(\frac{1}{\theta})=\frac{\pi}{2}\,,$ one obtains  \eqref{2 pi arctan}\,.
Consequently, \eqref{Pi(log((x2+a2)g(x)))} becomes
  \begin{align}\label{Pi(log((x2+a2)g(x)))+lemme}
    \Pi_+\Big(\log \big((y^2+a^2)^\ell g_{t,x}(y)\big) \Big)
    =&\,\frac{1}{2}\log\Big((y^2+a^2)^\ell g_{t,x}(y)\Big)-\frac{i\ell \,\pi}{2}+i\ell\,\mathrm{artan}\left(\frac{a}{y}\right)
    \notag
    \\
    -&\,\frac{i}{2\pi}\int_0^{+\infty}
    \frac{\log(g_{t,x}(y+s))-\log(g_{t,x}(y-s))}{s}\,\mathrm{d}s\,,   
\end{align}
and hence,
\begin{align}\label{exp(Pi(log((x2+a2)g(x))))}
    \exp\bigg(    \Pi_+\Big(\log \big((y^2+a^2)^\ell g_{t,x}(y)\big) \Big)\bigg)\Big|_{y=x}
    =&\,\eee^{\frac{-i\ell\, \pi}{2}}
    \sqrt{(x^2+a^2)^\ell\, g_{t,x}(x)} 
    \eee^{i\ell \, \mathrm{arctan}\left(\frac{a}{x}\right)}
    \\
    &\,\exp\left(-\frac{i}{2\pi}\int_0^{+\infty}
    \frac{\log(g_{t,x}(x+s))-\log(g_{t,x}(x-s))}{s}\,\mathrm{d}s\right)\,.\notag
\end{align}
Thus, combining \eqref{ZD-Q/prod}\,, \eqref{Q/prod} and \eqref{exp(Pi(log((x2+a2)g(x))))}\,
\begin{align}\label{ZD-gtx}
    \zd(t,x)=&\,\frac{u_0(x)}{\prod_{k=1}^\ell(x-y_{2k-1})}
    \frac{(x+ia)^\ell}{\sqrt{(x^2+a^2)^\ell}\, \mathrm{e}^{i\ell \arctan(a/x)}}\frac{\eee^{i \frac{\pi \ell}{2}}}{\sqrt{g_{t,x}(x)}}\notag
    \\[0.1cm]
    &\ \ \cdot \exp\left(\frac{i}{2\pi}\int_0^{+\infty}
    \frac{\log(g_{t,x}(x+s))-\log(g_{t,x}(x-s))}{s}\,\mathrm{d}s\right)\notag
    \\[0.1cm]
    =&\,\frac{\va{u_0(x)}}{\prod_{k=1}^\ell(x-y_{2k-1})}
    \frac{(i)^{\ell}}{\sqrt{g_{t,x}(x)}} \,\eee^{i\varphi (t,x)}\,,
\end{align}
where
\be\label{varphi}
    \varphi(t,x)=\arg(u_0(x))+\frac{1}{2\pi}\int_0^{+\infty}
    \frac{\log(g_{t,x}(x+s))-\log(g_{t,x}(x-s))}{s}\,\mathrm{d}s\,.
\ee
Substituting $g_{t,x}$ in \eqref{ZD-gtx} by its value in \eqref{g(y)} with $y=x$\,, 
\begin{align*}
    \zd(t,x)
    =&\,\frac{\va{u_0(x)}}{\prod_{k=1}^\ell(x-y_{2k-1})}
    \frac{(i)^{\ell}}{\displaystyle\sqrt{\frac{\mp 2t \va{u_0(x)}^2}{\displaystyle\prod_{k=0}^{2\ell}(x-y_k)}}} \,\eee^{i\varphi (t,x)}\,,
\end{align*}
and using the fact that the $y_k$ are solutions of the algebraic equation~$y_k\mp 2t \va{u_0(y_k)}^2=x$\,,  for all $k=0\,,\ldots\,,2\ell\,,$ we conclude
\begin{align*}
\zd(t,x)=&\,\frac{1}{(\mp 2t)^{\ell}\prod_{k=1}^\ell \va{u_0(y_{2k-1})}^2}
\frac{(i)^{\ell}}{\displaystyle\sqrt{\frac{\mp 2t} {\displaystyle(\mp2t)^{2\ell+1}\prod_{k=0}^{2\ell}\va{u_0(y_k)}^2}}}\eee^{i\varphi (t,x)}
\\[0.2cm]
=&\,
\left(\mp i\, \frac{\va{t}}{t}\right)^\ell \;\frac{\displaystyle\prod_{k=0}^{\ell}\va{u_0(y_{2k})}}{\displaystyle\prod_{k=1}^\ell\va{u_0(y_{2k-1})}} \ \ \mathrm{e}^{i \varphi(t,x)}\ ,
\end{align*}
where by \eqref{varphi} and \eqref{g(y)}\,,
\[
    \varphi(t,x)=\arg(u_0(x))+\frac{1}{2\pi}\int_0^{+\infty}
    \frac{1}{s}\log\left(\frac{\ s\mp 2t \va{u_0(x+s)}^2}{-s\mp 2t \va{u_0(x-s)}^2}\frac{\prod_{k=0}^{2\ell}(x-s-y_k)}{\prod_{k=0}^{2\ell}(x+s-y_k)}\right)\mathrm{d}s\,.
\]

\end{proof}

\begin{proof}[Proof of Lemma~\ref{Pi+}]
      For all $x\in\R\,,$
      \begin{align*}
          \lim_{\delta\to 0^+}\Pi h (x+i\delta)
          =&\,
          \lim_{\delta\to 0^+}\frac{1}{2\pi i}\int_{\R}\frac{h(y)}{y-x-i\delta}\,\mathrm{d}y
          \\
          =&
          \,
          \lim_{\delta\to 0^+} \frac{1}{2\pi} \int_{\R}\frac{\delta \,h(y)}{(y-x)^2+\delta^2}\,\mathrm{d}y
          +\frac{1}{2\pi i}\int_{\R}\frac{h(y)(y-x)}{(y-x)^2+\delta^2}\,\mathrm{d}y
      \end{align*}
      Using the Poisson integral formula on the upper half-plane of $\C\,,$ we infer for all $\delta>0\,,$
      \[
        \frac{1}{2\pi} \int_{\R}\frac{\delta\,h(y) }{(y-x)^2+\delta^2}\,\mathrm{d}y=\frac{h(x)}{2}\,.
      \]
      For the second term, observe that,
      \begin{align*}
          \lim_{\delta\to 0^+}\int_{\R}\frac{h(y)(y-x)}{(y-x)^2+\delta^2}\,\mathrm{d}y
          =&\,
          \lim_{\delta\to 0^+}\lim_{\eps\to 0}\int_{-\infty}^{x-\eps}\frac{h(y)(y-x)}{(y-x)^2+\delta^2}\,\mathrm{d}y
          \,+\,
          \lim_{\delta\to 0^+}\lim_{\eps\to 0}\int_{x+\eps}^{+\infty}\frac{h(y)(y-x)}{(y-x)^2+\delta^2}\,\mathrm{d}y
          \\
            =&\,\lim_{\delta\to 0^+}\lim_{\eps\to 0}\left( -\int_{\eps}^{+\infty}\frac{h(x-s)s}{s^2+\delta^2}\,\mathrm{d}s
            +
            \int_{\eps}^{+\infty}\frac{h(x+s)s}{s^2+\delta^2}\,\mathrm{d}s
            \right)
          \\
          =&\,\lim_{\delta\to 0^+}\int_{0}^{+\infty}\frac{(h(x+s)-h(x-s))\,s}{s^2+\delta^2} \;\mathrm{d}y\,,
      \end{align*}
      where we can switch $\lim$ and integral, as $\int_0^{+\infty}\frac{h(x+s)-h(x-s)}{s}\,\mathrm{d}s$ is well--defined as the principal value of $\,\frac{h(x)}{x}$ since $\|h'\|_{L^\infty}<\infty\,.$
  \end{proof}

\vskip0.25cm

 Now, armed with Proposition~\ref{Zd-solution multivalue prop}, let us proceed to prove the extended version of the statement of this proposition to Theorem~\ref{Theorem solution mutivalué- intro}.
 
 \begin{proof}[Proof of Theorem~\ref{Theorem solution mutivalué- intro}]
 Let $u_0\in \LtwoR\cap C^1(\R)\,,$ such that $u_0$ is tending to $0$ at infinity and $u_0'\in L^\infty(\R)\,$.  For such $u_0\,$,
 the $C^1$--function $\gamma_t(y):=y\mp2t \va{u_0(y)}^2$ is asymptotically equivalent to $y$ at $\pm \infty\,.$ Moreover, since $u_0'$ is bounded in $L^\infty(\R)$\,, 
 then for all $t\in\R\,,$ and for any $x\in\R$ that is not a critical value of $\gamma_t\,,$ the equation $\gamma_t(y)=x$ has a finite number of real solutions
 \be\label{sol simple}
    y_0(t,x)< \ldots< y_{2\ell}(t,x)\,.
 \ee
Besides, note that by the Sard theorem,  the set of critical values of $\gamma_t$ has zero Lebesgue measure. Moreover, since 
\[
\gamma_t'(y)=1+2t \big(\va{u_0(y)}^2\big)'\longrightarrow 1\,, \qquad y\to \pm\infty\,,
\]
then the set of critical points 
 $\lracc{y\;; \;\gamma_t'(y)=0}$ is compact for a given $t$\,, 
 so that its image --the set of critical values of $\gamma_t$-- is compact, and hence in particular closed. Thus,  let $\Omega$ be any open connected set (for the $x$ variable) where \eqref{sol simple} is satisfied.
The idea, at this stage of the proof, is to deduce the result from Proposition~\ref{Zd-solution multivalue prop}\,. Thus, by using a standard mollifier, we approximate $u_0$ in $L^2(\R)\cap C^1(\R)$ by a sequence of rational functions $(u_0^\delta)$ belonging to the Hardy space. 
   Now,  take $\omega$ to be any arbitrary open subset of $\Omega$ such that $\overline{\omega}$ is compact. Therefore, for $\delta$ small enough,
 \[
    y+2t \va{u_0^\delta(y)}^2=x\,,\qquad x\in\omega\,,
 \]
has $2\ell+1$ solutions
\[
    y_0^\delta(t,x)< \ldots< y_{2\ell}^\delta(t,x)\,.
\]
Since $u_0^\delta$ is a rational function in the Hardy space, then by Proposition~\ref{Zd-solution multivalue prop}\,, 
\[
 ZD_{\pm}[u_0^\del] (t,x)=\mathrm{e}^{i \varphi_\delta(x)}\left(\mp i\, \frac{\va{t}}{t}\right)^\ell \prod_{k=0}^{2\ell}\va{u_0^\delta(y_k^\delta)}^{(-1)^k} ,
\]
where 
\[
    \varphi_\delta(x)=\arg(u_0^\delta(x))+\frac{1}{2\pi}\int_0^{+\infty}
    \frac{1}{s}\log\left(\frac{\ s\mp 2t \va{u_0^\delta(x+s)}^2}{-s\mp 2t \va{u_0^\delta(x-s)}^2}\frac{\prod_{k=0}^{2\ell}(x-s-y_k^\delta)}{\prod_{k=0}^{2\ell}(x+s-y_k^\delta)}\right)\mathrm{d}s\,.
\]
By passing to the limit as $\delta\to 0\,,$ and using that $u_0^\del\to u_0$ in $L^2\cap C^1$ so that $y_k^\delta(t,x)\to y_k(t,x)\,,$ and by the weak limit of \eqref{cv faible du flot}\,, we deduce for every $x\in\omega\,,$ formula~\eqref{zd[u0]-multivalue solution-intro}\,. Now, since $\omega$ is chosen arbitrarily in $\Omega\,,$ this achieves the proof.

 \end{proof}

An immediate consequence of Theorem~\ref{Theorem solution mutivalué- intro} is the following corollary.

\begin{corollary}
    Let $u_0\in L^2_+(\R)\cap L^\infty(\R)$ (with $\|u_0\|_{L^2}<\sqrt{2\pi}$\, in the focusing case), then for all $t\in\R\,,$
    \[
    \|\zd(t)\|_{L^\infty}\leq \|u_0\|_{L^\infty}\,.
    \]
\end{corollary}

\begin{proof}
    In view of Theorem~\ref{Theorem solution mutivalué- intro}\,, we have for any $u_0\in \LtwoR\cap C^1(\R)\,,$ satisfying the property that $u_0$ is tending to $0$ at infinity\,,
    \[
           \va{\zd(t,x)}=\,\frac{ \displaystyle\prod_{k=0}^{\ell}\va{u_0(y_{2k})}}{\displaystyle\prod_{k=1}^\ell \va{u_0(y_{2k-1})}}\,,
       \]
    where $y_0 < \ldots < y_{2\ell}$ are solutions for the algebraic equation
       \[
            x-y_{k}=\mp 2t \va{u_0(y_k)}^2\,.
       \] 
       Therefore, by the monotonicity of $k\mapsto y_k$ we infer the monotonicity of $k\mapsto \va{u_0(y_k)}^2\,,$ so that we can deduce
       \[
            \va{\zd(t,x)}\leq \max\{\va{u_0(y_0)}\,, \va{u_0(y_{2\ell})}\}\leq \|u_0\|_{L^\infty}\,.
       \]
       The general case of $u_0 \in L^\infty(\R) \cap L^2(\R)$ follows by applying a standard mollifier to $u_0$ like the one described to prove Theorem~\ref{Theorem solution mutivalué- intro} and by using property \eqref{cv faible du flot} in Theorem~\ref{weak limit theorem zero dispersion}\,.
\end{proof}

\vskip0.5cm








\begin{thebibliography}{alpha}


\bibitem[Ba23a]{Ba23a}
\texttt{R.~Badreddine.}
    \newblock On the global well-posedness of the Calogero-Sutherland derivative nonlinear Schr\"odinger equation.
    \newblock{\em arXiv preprint arXiv:2303.01087 (2023). To appear in Pure and Applied analysis.}
    \newblock{\url{
https://doi.org/10.48550/arXiv.2303.01087}}

\bibitem[Ba23b]{Ba23b}
\texttt{R.~Badreddine.}
    \newblock Traveling waves \& finite gap potentials for the Calogero-Sutherland Derivative nonlinear Schr\"odinger equation.
    \newblock{\em arXiv preprint arXiv:2307.01592 (2023). To appear in  Annales IHP C, Analyse non linéaire. }
    \newblock{\url{
https://doi.org/10.48550/arXiv.2307.01592}}

\bibitem[CG09]{CG09}
\texttt{T.~Claeys and T.~Grava.}
    \newblock{ Universality of the break-up profile for the KdV equation in the small
dispersion limit using the Riemann-Hilbert approach.}
\newblock{\em  Communications in Mathematical Physics,
286(3):979–1009, (2009).}

\bibitem[CG10a]{CG10a}
\texttt{T.~Claeys and T.~Grava.}
    \newblock{  Painlevé II asymptotics near the leading edge of the oscillatory zone
for the Korteweg—de Vries equation in the small-dispersion limit.}
\newblock{\em  Communications on Pure and
Applied Mathematics, 63(2):203–232, (2010).}

\bibitem[CG10b]{CG10b}
\texttt{T.~Claeys and T.~Grava.}
    \newblock{Solitonic asymptotics for the Korteweg–de Vries equation in the small
dispersion limit.}
\newblock{\em  SIAM journal on mathematical analysis, 42(5):2132–2154, (2010).}

\bibitem[BLS08]{BLS08}
    \texttt{JL.~Bona\,, D.~Lannes and J-C. Saut.}
    \newblock \textit{Asymptotic models for internal waves.}
    \newblock{\em Journal de mathématiques pures et appliquées 89.6 (2008): 538-566.}


\bibitem[Ga23a]{Ga23a}
    \texttt{L.~Gassot.}
    \newblock Zero-dispersion limit for the Benjamin-Ono equation on the torus with single well initial data. 
    \newblock{\em arXiv:2111.06800, Communications in mathematical Physics 401, pp.2793–2843 } (2023).
    \newblock{\url{
https://doi.org/10.1007/s00220-023-04701-0}}

\bibitem[Ga23b]{Ga23b}
    \texttt{L.~Gassot.}
    \newblock  Lax eigenvalues in the zero-dispersion limit for the Benjamin-Ono equation on the torus. 
    \newblock{\em arXiv:2301.03919, To appear in SIAM J. Math. Analysis. } (2023).
    \newblock{\url{https://doi.org/10.1137/23M154635X}}


\bibitem[Ge22]{Ge22}
	\texttt{P.~G{\'e}rard.}
	\newblock An explicit formula for the Benjamin--Ono equation.
	\newblock {\em Tunisian Journal of Mathematics 5.3 (2023): 593-603.}
 \newblock{\url{https://doi.org/10.2140/tunis.2023.5.593}}

\bibitem[Ge23]{Ge23}
		\texttt{P.~G{\'e}rard.}
		\newblock The zero dispersion limit for the Benjamin--Ono equation on the line.
		\newblock {\em ArXiv preprint arXiv:2307.12768 } (2023). To appear in Compte Rendus-Série Mathématique.
  \newblock{\url{https://doi.org/10.48550/arXiv.2307.12768}}
  
\bibitem[GG15]{GG15}
		\texttt{P.~G{\'e}rard and S.~Grellier.}
		\newblock An explicit formula for the cubic Szegő equation.
		\newblock {\em Transactions of the American Mathematical Society, 367(4), 2979-2995. } (2015).
  \newblock{\url{
https://doi.org/10.48550/arXiv.1304.2619}}

\bibitem[GL22]{GL22}
    \texttt{P. Gérard, and E. Lenzmann.}
    \newblock  The Calogero--Moser Derivative Nonlinear Schr\"odinger Equation.
    \newblock{\em arXiv preprint arXiv:2208.04105 }
(2022). To appear in Comm. Pure Appl. Math.
\newblock{\url{
https://doi.org/10.48550/arXiv.2208.04105}}

\bibitem[GP23]{GP23}
    \texttt{ P. Gérard, and A. Pushnitski.}
    \newblock The cubic Szeg\H{o} equation on the real line: explicit formula and well-posedness on the Hardy class.
    \newblock {\em arXiv preprint arXiv:2307.06734} (2023).

\bibitem[GK07]{GK07}
    \texttt{T. Grava, C. Klein.}
  \newblock  Numerical solution of the small dispersion limit of Korteweg–de
Vries and Whitham equations.
    \newblock {\em Comm. Pure Appl. Math. 60, 1623–1664.} (2007).
 

\bibitem[HK24]{HK24}
    \texttt{J.~Hogan, and M.~Kowalski.}
    \newblock Turbulent Threshold for Continuum Calogero-Moser Models.
    \newblock{\em \textit{arXiv preprint arXiv:2401.16609.} }(2024).
    \newblock{\url{
https://doi.org/10.48550/arXiv.2401.16609
}}
    
\bibitem[KLV23]{KLV23}
    \texttt{R. Killip, T. Laurens, and M. Visan.}
    \newblock Scaling-critical well-posedness for continuum Calogero-Moser models.
    \newblock{\em \textit{arXiv preprint arXiv:2311.12334} }(2023).
    \newblock{\url{
https://doi.org/10.48550/arXiv.2311.12334
}}

\bibitem[LL83]{LL83}
    \texttt{P.~Lax and C.D.~Levermore.}
	\newblock The small dispersion limit of the Korteweg‐de Vries equation. I.
	\newblock {\em Communications on Pure and Applied Mathematics 36, no. 3,  pp.253-290} (1983).

\bibitem[MW16]{MW16}
    \texttt{P. Miller, A. Wetzel.}
    \newblock{The scattering transform for the Benjamin–Ono equation in the small dispersion limit}
    \newblock{\em Physica D 333, p. 185--199} (2016).
  
\bibitem[MX11]{MX11}
    \texttt{P.~Miller and Z.~Xu.}
	\newblock On the zero--dispersion limit of the Benjamin--Ono Cauchy problem for positive initial data.
	\newblock {\em Communications on pure and applied mathematics 64 no. 2, pp.205-270} (2011).

\bibitem[Ve87]{Ve87}
    \texttt{S.~Venakides.  }
    \newblock The zero dispersion limit of the Korteweg-de Vries equation with periodic initial data.
    \newblock{\em Transactions of the American Mathematical Society 301, no. 1 pp. 189-226 } (1987).

\bibitem[Ve91]{Ve91}
    \texttt{S.~Venakides. }
    \newblock The Korteweg-de Vries equation with  small dispersion limit: higher order LAx--Levermore theory. 
    \newblock{\em In Applied and Industrial Mathematics, p. 255--262. Springer, } (1991).
    
\bibitem[Pe95]{Pe95}
    \texttt{D.E.~Pelinovsky.}
    \newblock Intermediate nonlinear Schr\"odinger equation for internal waves in a fluid of finite depth,
    \newblock {\em Phys. Lett. A 197 (1995), no. 5–6, 401–406.}

\bibitem[Sa19]{Sa19}
    \texttt{J-C.~Saut.}
    \newblock Benjamin-Ono and intermediate long wave equations: Modeling, IST and PDE.
    \newblock {\em Nonlinear dispersive partial differential equations and inverse scattering (2019): 95-160.}
  
\end{thebibliography}
\end{document}